\documentclass[english]{ourlematema}
\usepackage{amsmath, amsthm, amssymb, hyperref, color}
\usepackage{MnSymbol} % for \dashedrightarrow command
\usepackage{graphicx}
\usepackage{caption}
\usepackage[all]{xypic}
\usepackage{verbatim}
\usepackage{chemfig, chemnum}
\usepackage{tikz}
\usepackage[linesnumbered,lined,commentsnumbered]{algorithm2e}
\usepackage{mathtools}
\usepackage{caption}
\usepackage{subcaption}
\usepackage{makecell}
\usepackage{array}

\newtheorem{theorem}{Theorem}
\numberwithin{theorem}{section}
\newtheorem{proposition}[theorem]{Proposition}
\newtheorem{lemm}[theorem]{Lemma}

\newtheorem{example}[theorem]{Example}
\newtheorem{conjecture}[theorem]{Conjecture}

\theoremstyle{definition}

\newcommand{\R}{\mathbb{R}}

\newcommand{\C}{\mathbb{C}}

\DeclareMathOperator{\Tr}{Tr}
\DeclareMathOperator{\spn}{span}

 \title{The degree of the central curve in semidefinite, linear, and quadratic programming}
 \titlemark{The degree of the central curve in SDP, LP, and QP}

  \author{Serkan Ho\c{s}ten}
  \address{%
  San Francisco State University \\
\email{serkan@sfsu.edu}
}

\author{Isabelle Shankar}
\address{%
University of California, Berkeley \\
\email{isabelle\_shankar@berkeley.edu}
}

\author{Ang\'{e}lica Torres}
\address{%
  KTH Royal Institute of Technology \\
\email{amtb@kth.se}
}

\date{2020/28/11}

\begin{document}
\maketitle
\begin{abstract}
\noindent 
The Zariski closure of the central path which interior point algorithms track in convex optimization problems such as linear, quadratic, and semidefinite programs is an algebraic curve. The degree of this curve
has been studied in relation to the complexity of these interior point algorithms, and  for linear programs
it was computed by De Loera, Sturmfels, and Vinzant in 2012. We show that the degree of the central curve for generic semidefinite programs is equal to the maximum likelihood degree of linear concentration models. New results from the intersection theory of the space of complete quadrics imply that this is a polynomial in the size of semidefinite matrices with degree equal to the number of constraints. Besides its degree we explore the arithmetic genus of the same curve. 
We also compute the degree of the central curve for generic linear programs with different techniques which extend to  bounding  the same degree for generic quadratic programs.
\end{abstract}

\section{Introduction}

Let $\mathcal{S}_\R^m$ and $\mathcal{S}_\C^m$ be the vector spaces
of $m \times m$ symmetric matrices with real and complex entries, respectively. Our starting point is semidefinite programs of the form

\begin{equation}\label{eq:SDP_primal}
    \begin{split}
        \text{minimize } & \langle C, X\rangle\\
        \text{subject to } & \langle A_i, X \rangle = b_i, \quad i=1,\ldots ,d\\
        & X\succeq 0
    \end{split}
\end{equation}
where $C$ and $A_i, i=1, \ldots, d$, are in $\mathcal{S}_\R^m$, and $b_i \in \R$ for $i=1, \ldots, d$. We use the standard Euclidean
inner product $\langle Y, Z \rangle = \Tr(YZ)$ on $\mathcal{S}_\R^m$,
and $X \succeq 0$ means that $X$ belongs to the cone of $m \times m$ positive semidefinite matrices. Typically, we will assume that 
the cost matrix $C$, the constraint matrices $A_1, \ldots, A_d$, and 
$b = (b_1, \ldots, b_d)^t$ are generic. This assures, among other things, that if  \eqref{eq:SDP_primal} is feasible, it is strictly feasible.

The central curve of the above semidefinite program is obtained from the Karush-Kuhn-Tucker (KKT) conditions to an auxiliary optimization problem with a logarithmic barrier function. The KKT conditions are

\begin{equation}\label{eq:KKT_conditions_SDP}
    \begin{split}
        &C-\lambda X^{-1}-\Sigma _{i=1}^d y_iA_i =0,\\
        &\langle A_i, X \rangle =b_i, \quad i=1,\ldots ,d, \\
        &X\succeq 0
    \end{split}
\end{equation}
where $y_1,\ldots,y_d$ are the dual variables to the dual semidefinite program. 
\begin{dfn} \label{dfn:SDPcentralcurve}
    Let $(X^*(\lambda),y^*(\lambda))$ be the unique solution of the system~\eqref{eq:KKT_conditions_SDP} for a fixed $\lambda > 0$. The (primal) \emph{central curve} $\mathcal{C}_{SDP}(C,\{A_i\},b)$ is the projection onto $\mathcal{S}_\C^m$ of the Zariski closure in $\mathcal{S}_\C^m \times \C^d$
    of $\{(X^*(\lambda),y^*(\lambda)): \lambda > 0\}$.
\end{dfn}

The central curve contains the \emph{central path} $\{X^*(\lambda): \lambda > 0\}$. Interior point algorithms follow a piecewise linear approximation to the central path 
to obtain an optimal solution to \eqref{eq:SDP_primal} as $\lambda$ approaches zero  \cite{BV2004,FM1968,FGW2002,NW2006,NN1994}. The degree of $\mathcal{C}_{SDP}(C,\{A_i\},b)$ can be used to give an upper bound on the total curvature of the central path which is a heuristic measure on the number of steps interior point algorithms will take to find an optimal solution. 
%The degree of $\mathcal{C}_{SDP}(C,\{A_i\},b)$ is a measure of the %"curvy-ness" of the central path and sheds light on the number of steps
%interior point algorithms will take to find an optimal solution. 

Interior point methods were first developed for linear programming problems, and the study of the central curve for linear programming from the perspective of algebraic geometry was initiated by Bayer and Lagarias in \cite{BL89-1} and \cite{BL89-2}. Dedieu, Malajovich, and Shub \cite{DMS05} studied the total curvature of the central path for linear programs in relation to bounding the number of iterations interior point algorithms take. By now we know that the total curvature can be exponential in the dimension of the ambient space \cite{ABGJ18}. Most relevant to our work, De Loera, Sturmfels, and Vinzant \cite{DSV12} obtained a breakthrough by computing the degree of the linear programming central curve. 
Given the linear program
\begin{equation}\label{eq:LP_primal}
    \begin{split}
        \text{minimize } & cx\\
        \text{subject to } & Ax=b\\
        & x \geq 0,
    \end{split}
\end{equation}
where $c \in \R^m$ is a row vector, $A$ is $d \times m$ matrix of rank $d$, and $b \in \R^d$ is a column vector, they have related this degree to the degree of a {\it reciprocal variety} and a matroid invariant. 
\begin{thm} \cite[Lemma 11]{DSV12} For generic $b$ and $c$, the degree
of the central curve of the linear program \eqref{eq:LP_primal} is 
equal to the degree of the reciprocal variety {\small 
$$\mathcal{L}_{A,c}^{-1} := \overline{ \left \{  (u_1, \ldots, u_m) \in \C^m \,: \, 
\left(\frac{1}{u_1}, \ldots ,\frac{1}{u_m}\right) \in \mathrm{rowspan}\left( \begin{array}{c} A \\ c \end{array} \right) \mbox{ and } u_i \neq 0, i=1, \ldots, m \right \}}  $$ }
as well as the M\"obius number $|\mu(A,c)|$ of the rank $d+1$ 
matroid associated to the row span of 
$\left( \begin{array}{c} A \\ c \end{array} \right)$. When $A$ is also
generic, the degree of the central curve is equal to 
$$ \binom{m-1}{d}.$$
\end{thm}

Our main contribution is Theorem \ref{SDP-theorem} where we prove that the degree of the central curve for the SDP \eqref{eq:SDP_primal}
when $C$ and $b$ are generic is equal to the maximum likelihood degree (ML degree) of the linear concentration model generated by $\{A_i\}$ and $C$. When $\{A_i\}$ are also generic, this degree is equal to the degree of the reciprocal variety associated to the linear subspace
$\mathcal{L}_{\{A_i\}, C} = \spn \{A_1, \ldots, A_d, C\}$:
$$ \mathcal{L}_{\{A_i\},C}^{-1} := \overline{ \left \{ X \in \mathcal{S}_\C^m \, : \, X^{-1} \in \mathcal{L}_{\{A_i\}, C} \right \}}.$$
We further show in Corollary \ref{cor:SDP-symmetry} that, when $\{A_i\}, C$, and $b$ are generic, the degree of 
$\mathcal{C}_{SDP}(C, \{A_i\}, b)$ is symmetric in the number of 
the linear equations defining \eqref{eq:SDP_primal}. Corollary \ref{cor:SDP-polynomial} concludes that in this case the degree of the central curve is a polynomial in $m$ of degree $d$.  This theorem and the two corollaries complete the work started in \cite{Rhodes}, proving Conjectures 4.3 and 4.4 in the same work. 

In the remainder of Section \ref{sec:SDP} we will report our observations on the arithmetic genus of $\mathcal{C}_{SDP}(C, \{A_i\}, b)$. We will also discuss semidefinite programs and the degree of their central curves associated to sum of squares (SOS) polynomials. In Section \ref{sec:LP} we will revisit the degree of the central curve of the linear program 
\eqref{eq:LP_primal} when $A, c$, and $b$ are generic. Besides relating this degree to the ML degree of linear concentration models generated
by diagonal matrices, in Theorem \ref{LP-volume} we will provide a different proof that this degree is equal to $\binom{m-1}{d}$. Section \ref{sec:QP} extends this result and its proof technique to convex
quadratic programs with linear constraints. Theorem \ref{QP-volume} bounds the degree of the central curve of such programs when the objective function and the constraints are generic.

\section{Semidefinite Programs and Linear Concentration Models}
\label{sec:SDP}

In this section we consider the central curve $\mathcal{C}_{SDP}(C,\{A_i\},b)$ when $C$ and $b$ are generic. 
In what follows, we describe the degree of this curve as the 
ML degree of a linear concentration model.  When $\{A_i\}$ are
also generic, we denote $\deg \left(\mathcal{C}_{SDP}(C, \{A_i\},b) \right)$ by $\psi_{SDP}(m,d)$. 

\subsection{Linear concentration models and  degree of the central curve}

Let $\mathcal{L}$ be a linear subspace of $\mathcal{S}^m_\R$ spanned by $d$ linearly independent symmetric matrices $\{K_1,\ldots ,K_d\}$. A linear concentration model is the set 
\[\mathcal{L}_{\succeq 0}^{-1}\coloneqq \{\Sigma \in \mathcal{S}_{\succeq 0}^m \, : \, \Sigma^{-1}\in\mathcal{L} \} \]
where $\mathcal{S}_{\succeq 0}^m$ is the cone of positive semidefinite matrices. Every matrix $\Sigma$ in $\mathcal{L}_{\succeq 0}^{-1}$ is the covariance matrix of a multivariate normal distribution on $\R^m$, and the elements of $\mathcal{L}$ are concentration matrices.

Given a sample covariance matrix $S$, the maximum likelihood estimate $\hat{K}$ of $S$ with respect to the linear concentration model defined by $\mathcal{L}$ is the unique positive semidefinite solution to the zero-dimensional polynomial equations 
\begin{equation}\label{eq:MLE}
    \Sigma K=Id_m, \qquad K\in\mathcal{L}, \qquad \Sigma -S\in \mathcal{L}^\perp.  
\end{equation}
The ML degree of this linear concentration model is defined as the number of solutions to \eqref{eq:MLE} in $\mathcal{S}_\C^m$. 

In~\cite{SU10} it was proven that when the matrices $K_1,\ldots ,K_d$ are generic, the ML degree of the linear concentration model is precisely the degree of the reciprocal variety $\mathcal{L}^{-1}$. %defined as 
%\[\mathcal{L}^{-1}\coloneqq \overline{ \{X\in\mathcal{S}^m : %X^{-1}\in\mathcal{L}\}}.\]
%Note that $\mathcal{L}^{-1}_{\succeq 0}= \mathcal{L}^{-1}\cap %\mathcal{S}^m_{\succeq 0}.$

\begin{thm}\cite[Theorem 2.3]{SU10}\label{thm:ML_degree}
    The ML degree $\phi(m,d)$ of a linear concentration model defined by a generic linear subspace $\mathcal{L}$ of dimension $d$ in $\mathcal{S}^m$ equals the degree of the projective variety $\mathcal{L}^{-1}$. This degree further  satisfies 
    \[\phi(m,d)=\phi\left( m, {m+1\choose 2} +1-d\right).\]
\end{thm}

Now we are ready to prove our main theorem.

%\color{red}
%\begin{itemize}
%    \item Say we connect $\psi_{SDP}(m,d)$ to the ML degree
%    of linear concentration models. 
%    \item Briefly describe: linear concentration models and ML degree
%    \item Define reciprocal variety of a subspace in general.
%\end{itemize}
%\color{black}
\begin{thm}\label{SDP-theorem}
Given an SDP as in~\eqref{eq:SDP_primal} with $C$ and $b$ generic, 
$\deg\left( \mathcal{C}_{SDP}(C,\{A_i\}, b) \right)$ is equal to
 the ML degree of the linear concentration model generated
 by $\mathcal{L} = \spn \{C, A_1, \ldots, A_d\}$.
%$\psi_{SDP}(m,d)=\phi(m,d+1)$, .  
% the ML degree of generic linear concentration models generated by $d+1$ symmetric matrices in $\mathcal{S}_\R^m$ \cite{SU10}.
If in addition $A_1,\ldots , A_d$ are generic, $\psi_{SDP}(m,d)$ is equal to the degree of $\mathcal{L}^{-1}$, and hence $\psi_{SDP}(m,d)=\phi(m,d+1)$.
\end{thm}
\begin{proof}
By definition  $$\deg(\mathcal{C}_{SDP}(C,\{A_i\},b))=\left| \mathcal{C}_{SDP}(C,\{A_i\},b) \cap \mathcal{H}\right|$$ where $\mathcal{H}$ is a generic hyperplane in $\mathcal{S}_\C^m$. 
Using the KKT conditions \eqref{eq:KKT_conditions_SDP}, the equations defining $\mathcal{C}_{SDP}(C,\{A_i\},b) \cap \mathcal{H}$ are 
\begin{equation}\label{eq:SDP_degree}
    \begin{split}
        &X^{-1}=\frac{1}{\lambda}C-\frac{1}{\lambda}\Sigma _{i=1}^d y_iA_i \\
        &\langle A_i,X\rangle -b_i=0, \quad i=1,\ldots ,d \\
        &\langle B,X \rangle -b_{d+1}=0,
    \end{split}
\end{equation}
for some generic $B\in\mathcal{S}_\C^m$ and $b_{d+1}\in\C$.

The first equation in~\eqref{eq:SDP_degree} means that $X^{-1}\in\mathcal{L}$, where $\mathcal{L}=\spn\{C,A_1,\ldots ,A_d\}$. Since $C$ is generic, in the last equation of~\eqref{eq:SDP_degree} we can take $B=C$. Additionally, if we define $S$ as a matrix such that $\langle A_i,S\rangle =b_i$, for $i=1,\ldots ,d$, and $\langle C,S\rangle =b_{d+1}$, the last $d+1$ equations in~\eqref{eq:SDP_degree} mean that $X-S \in \mathcal{L}^\perp$. Note that these are precisely the likelihood equations of the linear concentration model determined by $\mathcal{L}$. This proves that $\deg(\mathcal{C}_{SDP}(C,\{A_i\},b))$ is equal to the ML degree of
the linear concentration model defined by $\mathcal{L}$.
Additionally, if $A_1,\ldots ,A_d$ are generic, Theorem~\ref{thm:ML_degree} guarantees that $\phi(m,d+1)$ coincides with the degree of $\mathcal{L}^{-1}$, which means that $\psi(m,d)$ is equal to the degree of $\mathcal{L}^{-1}$ as well.

\end{proof}
\begin{cor} \label{cor:SDP-symmetry}
The degree of the central curve for a generic SDP satisfies
$$\psi_{SDP}(m,d)  = \psi_{SDP}\left( m, {m+1 \choose 2} - d - 1 \right).$$
\end{cor}
\begin{proof}
\begin{align*}
    \psi_{SDP}(m,d)  &= \phi(m,d+1)\\
    &= \phi(m,\binom{m+1}{2}+1-(d+1))\\
    &= \phi(m,\binom{m+1}{2}-d))\\
    &=\psi_{SDP}\left( m, {m+1 \choose 2} - d -1 \right).
\end{align*}
%The second equality follows from the symmetry of the bidegree 
%of Rees algebra.
\end{proof}

\begin{cor} \label{cor:SDP-polynomial}
$\psi_{SDP}(m,d)$ is a polynomial in $m$ of degree $d$.
\end{cor}
\begin{proof}
This result follows from the work of Micha\l{}ek, Monin, Wi\'{s}niewski, Manivel, Seynnaeve, and Vodi\v{c}ka
%Seynnaeve and collaborators 
who employed the space of complete quadrics and intersection theory
to prove the polynomiality of $\phi(m,d)$ (\cite{MMW2020} and \cite[Theorem 1.3]{MMMSV2020})
and from the seperate work of Cid-Ruiz \cite[Corollary C]{C2020}.
\end{proof}

%\subsection{Equations for the SDP central curve}

\subsection{Arithmetic Genus}

The ideal of polynomials $I_{\mathcal{L}_{\{A_i\},C}^{-1}}$ in $\C[x_{ij} \, : \, 1 \leq i \leq j \leq m]$ vanishing on the reciprocal variety $\mathcal{L}_{\{A_i\},C}^{-1}$ is a prime ideal since this variety is irreducible.
The proof of Theorem \ref{thm:ML_degree}  (see \cite[Theorem 2.3]{SU10})
relies on the fact that $I_{\mathcal{L}_{\{A_i\},C}^{-1}}$ is Cohen-Macaulay  when $\{A_i\}$ and $C$ are generic \cite{HVV85, Kot91}.
Since the central curve $\mathcal{C}_{SDP}(C, \{A_i\}, b)$ is obtained
from intersecting the reciprocal variety with $d$ generic linear equations in \eqref{eq:KKT_conditions_SDP}, the numerator
of the Hilbert series of $I_{\mathcal{L}_{\{A_i\},C}^{-1}}$
and that of the defining ideal of the the central curve are identical.
The Hilbert series for the central curve will be of the form 
$$\frac{h_0 + h_1t + h_2t^2 +\cdots + h_kt^k}{(1-t)^2}$$
where the coefficients $h_j$ are nonnegative integers with $h_0 = 1$ and $h_k \neq 0$.
The arithmetic genus of the central curve can be calculated as
$$\mathrm{genus}(m,d) := \mathrm{genus}(\mathcal{C}_{SDP}(C,\{A_i\}, b)) = 1 - \sum_{j=0}^k(1-j)h_j.$$
The following table shows $\mathrm{genus}(m,d)$ for all values 
we can compute with Macaulay2 \cite{M2} and/or using the two propositions that follow. 
%Numerator of the Hilbert Series
%
%\begin{center}
%\begin{tabular}{ c| c c c c c }
% m \textbackslash \, d & 1 & 2 & 3 & 4 & 5 \\ \hline
% 2 & 1  &   &  &  &\\  
% 3 & $1+t$  & $1+3t$  & $1+2t+t^2$ & $1+t$ &  \\
% 4 &  $1+2t$ & $1+7t+t^2$  & $1+6t + 10t^2$ & $1+5t+10t^2+5t^3$ & %$1+4t+10t^2+6t^3$ \\
% 5 & $1+3t$ & $1+12t+3t^2$ &  &  &
%\end{tabular}
%\end{center}
%
%\begin{center}
%\begin{tabular}{ c| c c c c}
% m \textbackslash \, d & 6 & 7 & 8  & 13 \\ \hline
% 4 & $1+3t+6t^2+7t^3$ & $1+2t+3t^2+2t^3+t^4$ & $1+t+t^2$ & \\
% 5 &  &  &  &   $1+t+t^2+t^3$ 
%\end{tabular}
%\end{center}

$$
\begin{array}{ c| c c c c c c c c c c c c c c}
 m \backslash \, d & 1 & 2 & 3 & 4 & 5 & 6 & 7 & 8 & 9 & 10& 11 & 12 & 13 & 14 \\ \hline
 2 & 0  & 0  &  &  &\\  
 3 & 0  & 0  & 1 & 0 &  0\\
 4 &  0 &  1 & 10 & 20 & 22 & 20 & 10 & 1 & 0 \\
 5 &  0 & 3 &  &  &  &  &  &  &  &  & & 33  & 3 & 0
\end{array}
$$
\begin{proposition} For $m \geq 2$, 
$$\mathrm{genus}(m,1) = \mathrm{genus}\left(m, \binom{m+1}{2} - 1\right) = 0.$$
In these cases, the central curve is a rational curve. Furthermore,
when $d=1$ the numerator of the Hilbert series is $1+(m-2)t$, and 
when $d = \binom{m+1}{2} - 1$ it is $1$.
\end{proposition}
\begin{proof}
In the case $d = \binom{m+1}{2}-1$, the reciprocal variety is equal
to $\mathbb{P}^d$, and therefore the central curve is $\mathbb{P}^1$.
In the case $d=1$, the reciprocal variety is the image of $\spn \{C, A_1\} \simeq \mathbb{P}^1$ under the rational map given by the $(m-1)$-minors
of a generic $m \times m$ symmetric matrix. Hence it is a rational 
curve of degree $m-1$. This implies that the numerator of the Hilbert
series of the ideal defining the reciprocal variety, and therefore that
of the central curve, is $1 + (m-2)t$. This means that the central 
curve is also a rational curve, i.e., its genus is equal to zero.
\end{proof}
\begin{proposition} \label{prop:codim1-2-formula} For $m \geq 2$, 
$$ \mathrm{genus}\left(m, \binom{m+1}{2}-2\right) = \binom{m-2}{2} \mbox{ and } 
\mathrm{genus}\left(m, \binom{m+1}{2}-3\right) = 1 + (m-1)^2(m-3).$$
\end{proposition}
\begin{proof}
In the first case, the reciprocal variety is a hypersurface defined
by a single polynomial of degree $m-1$. Therefore the numerator 
of the Hilbert series is equal to $1+t + \cdots + t^{m-2}$. Therefore
the arithmetic genus of the central curve is 
$$1 - \sum_{j=0}^{m-2} (1-j) = \sum_{j=1}^{m-3} j = \binom{m-2}{2}.$$
In the second case, the reciprocal variety is of codimension two, and 
it is a complete intersection generated by two degree $m-1$
generators; see \cite[p. 611]{SU10} and  Lemma \ref{lem:codim-2} below. Therefore the numerator of the Hilbert series is equal to 
$(1+t+ \cdots + t^{m-2})^2 = 1 + 2t + \cdots + (m-2) t^{m-3} + (m-1) t^{m-2} + (m-2)t^{m-1} + \cdots + 2 t^{2m-5} + t^{2m-4}$. Using
the formula for the arithmetic genus first yields 
$1 + (2m-6) \binom{m-1}{2} + (m-3)(m-1)$. This in turn is equal to 
$1 + (m-3)(m-1)^2$.
\end{proof}

\begin{lemma} \label{lem:codim-2}
When $d = \binom{m+1}{2}-3$, the reciprocal variety $\mathcal{L}_{\{A_i\},C}^{-1}$ associated to a generic linear subspace $\mathcal{L}_{\{A_i\}, C}$ is a complete intersection of codimension two generated by two polynomials of degree $m-1$.
\end{lemma}
\begin{proof}
Let $V$ be the variety of codimension $3$ in $\mathbb{P}^{\binom{m+1}{2}-1}$ defined by the $(m-1)$-minors of a generic $m \times m$ symmetric matrix, and let $X$ be the quasiprojective variety $\mathbb{P}^{\binom{m+1}{2}-1} \setminus V$. Consider the regular map 
$F \, : \, X \longmapsto \mathbb{P}^{\binom{m+1}{2}-1}$ given by the $(m-1)$-minors of a generic $m \times m$ symmetric matrix. Given the generic codimension two subspace $\mathcal{L}_{\{A_i\}, C}$, the inverse
image $F^{-1}(\mathcal{L}_{\{A_i\}, C})$ is an irreducible subvariety
of $X$ by Bertini's theorem \cite[Theorem 3.3.1]{Lazarsfeld}. This subvariety is defined by two generic 
linear combinations of $(m-1)$-minors, $f_1$ and $f_2$, which are of degree $m-1$. The variety in $\mathbb{P}^{\binom{m+1}{2}-1}$ defined by the same two polynomials is a complete intersection of codimension two. This variety contains the reciprocal variety which is irreducible and has also codimension two. Therefore if the ideal $\langle f_1, f_2 \rangle$ is prime it has to be the defining ideal of the reciprocal variety. But this is the case, since it is a complete intersection and hence all its  components have the same codimension. Any component other than the one coming from $F^{-1}(\mathcal{L}_{\{A_i\}, C})$ is associated to $V$, but
$V$ has codimension three. 
\end{proof}

We note that in the above table the entry for $m=5$ and $d=12$ is
computed using Proposition \ref{prop:codim1-2-formula}. However, 
the entry for $m=5$ and $d=3$, which is conjecturally equal to $33$ is
missing. Nevertheless, we venture to state the following conjecture.
\begin{conjecture} $\mathrm{genus}(m, d) = \mathrm{genus}\left(m, \binom{m+1}{2} - d\right)$.
\end{conjecture}

Although we cannot prove this conjecture, we can prove the analogous
statement for the central curve of linear programs \eqref{eq:LP_primal}
when $A$, $c$ and $b$ are generic. The central curve for linear
programs is defined as in Definition \ref{dfn:SDPcentralcurve} but
using the KKT conditions for linear programs; see \eqref{eq:KKT_conditions_LP} below.
\begin{thm}
    Let $A_d$ and $A_{m-d}$ be generic matrices of size $d\times m$ and $(m-d)\times m$ and of rank $d$ and $m-d$, respectively. Let $b_d$ and $b_{m-d}$ be two generic vectors in $\R^d$ and $\R^{m-d}$. The central curve of the linear program defined by $A_d,b_d$, and a generic vector $c$ has the same arithmetic genus as the central curve of the linear program defined by $A_{m-d},b_{m-d}$ and $c$.  
\end{thm}
\begin{proof}
    Let $\mathcal{C}_{LP}(d)$ and $\mathcal{C}_{LP}(m-d)$ denote the central curve of the generic linear programs as in the statement. 
    In this generic case, from~\cite{DSV12} we have 
    \begin{equation}\label{eq:LP_genus_d}
        \mathrm{genus}(\mathcal{C}_{LP}(d))=1-\sum_{j=0}^d(1-j){m-d+j-2\choose j},
    \end{equation}
    \begin{equation}\label{eq:LP_genus_m-d}
        \mathrm{genus}(\mathcal{C}_{LP}(m-d)) =1-\sum_{j=0}^{m-d}(1-j){d+j-2\choose j},
    \end{equation}
    where the binomial coefficients in each equation come from the coefficients of the Hilbert series computed in~\cite{DSV12}. To check that both computations have the same value, we need the identities
    \begin{equation*}
        \sum_{j=0}^{n}{r+j \choose j}={r+n+1 \choose n} \quad \mbox{ and } \quad \sum_{j=0}^{d} j{m-d+j-2 \choose j}=(m-d-1){m-1\choose d-1}.
    \end{equation*}
    First we get
    \begin{equation*}
        \begin{split}
            \mathrm{genus}(\mathcal{C}_{LP}(d))=&1-\sum_{j=0}^d{m-d+j-2\choose j}+\sum_{j=0}^d j{m-d+j-2\choose j}\\[0.1cm]
            =& 1- {m-d-2+d+1\choose d} + (m-d-1){m-1\choose d-1} \\[0.1cm]
            =& 1- {m-1\choose d}+ (m-d-1){m-1\choose d-1}\\[0.1cm]
            =& 1-\frac{(m-1)!}{(m-1-d)!d!}+(m-d-1)\frac{(m-1)!}{(m-1-d+1)!(d-1)!}\\[0.1cm]
            =& 1-\frac{(m-1)!(m-md+d^2)}{(m-d)!d!}
        \end{split}
    \end{equation*}
    where the second line comes from the identities mentioned above with $r=m-d-2$.  Doing a similar computation for $\mathrm{genus}(\mathcal{C}_{LP}(m-d))$ we get
    \begin{equation*}
         \begin{split}
            \mathrm{genus}(\mathcal{C}_{LP}(m-d))=&1-\sum_{j=0}^{m-d}{m-(m-d)+j-2\choose j}+\sum_{j=0}^{m-d} j{m-(m-d)+j-2\choose j}\\[0.1cm]
            =& 1- {d-2+m-d+1\choose m-d} + (d-1){m-1\choose m-1-d} \\[0.1cm]
            =& 1- {m-1\choose m-d}+ (d-1){m-1\choose m-1-d}\\[0.1cm]
            =& 1-\frac{(m-1)!}{(d-1)!(m-d)!}+(d-1)\frac{(m-1)!}{d!(m-1-d)!}\\[0.1cm]
            =& 1-\frac{(m-1)!(m-md+d^2)}{(m-d)!d!}.
        \end{split}
    \end{equation*}
\end{proof}

\subsection{Sum of Squares Polynomials}

We conclude Section~\ref{sec:SDP} by considering semidefinite programs
for sums of squares problems. For this, let $p \in \R[x_1,\ldots,x_n]$ be a homogeneous polynomial of degree $2D$ and let $L_p$ be the affine subspace of symmetric matrices $Q$ satisfying the identity
\begin{equation}\label{equ:mQm}
    p = [x]^T Q [x]
\end{equation}
where $[x]$ is a vector of all monomials of degree $D$ in $n$ variables.
The intersection of $L_p$ with the cone of positive semidefinite matrices is the Gram spectrahedron of $p$, and it is nonempty if and only if $p$ is a sum of squares (SOS) polynomial.
That is, certifying that a polynomial is SOS reduces to checking 
the feasibility of an SDP. This can be achieved by solving an SDP
using a random (generic) cost matrix $C$.

\begin{example}
 Suppose we wish to show that a generic ternary quartic is an SOS.
 The $A_i$'s and $b_i$'s come from equating coefficients in  (\ref{equ:mQm}).
 For example, if we let $[x] = [x^2, xy, xz, y^2, yz, z^2]$, the linear equation for the $x^2y^2$ term will be
 \begin{equation*}
     p_{(2,2,0)} = \langle A_{(2,2,0)}, Q \rangle
 \end{equation*}
where $ p_{(2,2,0)}$ is the $x^2y^2$ coefficient of the random ternary quartic $p$, 
$$ A_{(2,2,0)} = 
\begin{bmatrix}
0 & 0 & 0 & 1 & 0 & 0 \\
0 & 1 & 0 & 0 & 0 & 0 \\
0 & 0 & 0 & 0 & 0 & 0 \\
1 & 0 & 0 & 0 & 0 & 0 \\
0 & 0 & 0 & 0 & 0 & 0 \\
0 & 0 & 0 & 0 & 0 & 0 \\
\end{bmatrix}
$$
and $Q$ is the decision variable of the SDP which will have $d= \binom{3+4-1}{4} = 15$ constraints with matrices of size $m=\binom{3+2-1}{2} = 6$.
\end{example}

In general, the matrices $\{A_i\}$ for the linear constraints will 
be sparse and far from generic. However, if the polynomial $p$ 
that we want to certify to be an SOS polynomial is generic, then 
the $b_i$'s in the corresponding SDP will be also generic. Using a generic
cost matrix $C$ in this SDP allows us to consider the degree of 
the central curve for a generic SOS polynomial. 

We wish to report our computations in three instances: binary
sextics $(n=2, 2D=6)$, binary octics $(n=2, 2D=8)$, and ternary quartics $(n=3, 2D=4)$. The corresponding SDPs are given by input data with $m = 4, d=7$ for binary sextics, $m=5, d=9$ for binary octics,
and $m=6, d=15$ for ternary quartics. We note that for the same size
SDPs with generic $\{A_i\}$ we will obtain $\psi_{SDP}(4,7) = 9$,
$\psi_{SDP}(5,9) = 137$, and $\psi_{SDP}(6,15) = 528$. We believe
that studying this invariant for various families of SOS polynomials
is an interesting future project.
%{\color{blue} Should we use the notation $\psi$? We had mentioned before that this notation was for the case when everything was generic.}
\begin{proposition}
The degrees of the central curves for SDPs associated to generic binary sextics, binary octics, and ternary quartics, where generic cost matrices are used, are $7$, $45$, and $66$, respectively. 
\end{proposition}

%Below we record computational results that suggest what the degree of %the central curve for SOS programs are.
%\begin{center}
%\begin{tabular}{ c| c c c }
% n \textbackslash \, 2D & 4 & 6 & 8 \\ \hline
% 2 &  1(1) & 7(9)  & 45(137)\\  
% 3 & 66(528)  &   & \\
% 4 &   &   &
%\end{tabular}
%\end{center}
%In parenthesis we provide degree in the generic case for comparison.

%\subsection{Symmetric Polynomials as Sums of Squares}
%\begin{center}
%\begin{tabular}{ c| c c c }
% n \textbackslash \, 2D & 4 & 6 & 8 \\ \hline
% 2 &  1(1) & 3(3)  & ()\\  
% 3 & ()  &   & \\
% 4 &   &   &
%\end{tabular}
%\end{center}
As a last remark about the SDP arising from sums of squares, we would like to mention that since $C$ and $b$ are generic, the computations in the previous proposition, also correspond to the ML degree of a linear concentration model. Namely, the concentration model defined by catalecticants and an additional generic matrix corresponding to the cost matrix. Exploring this relation is also an interesting future project.
\section{Linear Programs} \label{sec:LP} 

By choosing $C$ and $\{A_i\}$ in \eqref{eq:SDP_primal} to be 
diagonal matrices we recover linear programs \eqref{eq:LP_primal}.
%\begin{equation}\label{eq:LP_primal}
%    \begin{split}
%        \text{minimize } & c^tx\\
%        \text{subject to } & Ax=b\\
%        & x \geq 0,
%    \end{split}
%\end{equation}
%where $c \in \R^m$, $A$ is $d \times m$ matrix of rank $d$, and $b \in %\R^d$.  
The central curve $\mathcal{C}_{LP}(c,A,b)$ for such a linear program can be defined as in the case of the central curve for a semidefinite program using the corresponding KKT conditions:

\begin{equation}\label{eq:KKT_conditions_LP}
    \begin{split}
        &c-\lambda \left(\frac{1}{x_1},\ldots, \frac{1}{x_m}\right)- y^tA =0\\
        & Ax=b, \\
        & x \geq 0.
    \end{split}
\end{equation}
When in the data defining (\ref{eq:LP_primal}), $c$ and $b$ are generic
the degree of the central curve $\mathcal{C}_{LP}(c,A,b)$ is 
equal to the degree of the reciprocal variety $\mathcal{L}_{A,c}^{-1}$ 
\cite[Lemma 11]{DSV12}. Further, if $A$ is also generic, this 
degree is equal to $\binom{m-1}{d}$. For the case when all the data is generic, we will denote the degree of the linear programming central curve by $\psi_{LP}(m,d)$.

%\begin{thm} \label{LP-theorem} \cite[Theorem 13]{DSV12} When $c$, $A$, %and $b$ are 
%generic, the degree of the central curve is 
%$$\psi_{LP}(m,d) \quad = \quad {m-1 \choose d}$$
%\end{thm}
The observations that connect the ML degree of generic linear
concentration models to the degree of the central curve of generic 
semidefinite programs have their counterpart here as well. One can consider the ML degree of linear concentration models generated by diagonal matrices as in \cite[Section 3]{SU10}. For generic models
we denote the ML degree by $\phi_{\mathrm{diag}}(m,d)$. A consequence
of Corollary 3 in \cite{SU10} is the following.
\begin{cor}
$$ \phi_{\mathrm{diag}}(m,d) \quad = \quad {m-1 \choose d-1}.$$
\end{cor}

An argument parallel to the one used in the proof of Theorem \ref{SDP-theorem} gives
\begin{cor}
$\psi_{LP}(m,d) = \phi_{\mathrm{diag}}(m,d+1)$.
\end{cor}

In the rest of this section we will develop another method to prove
that $\psi_{LP}(m,d) = \binom{m-1}{d}$. This method will be extended
for bounding the degree of the central curve for generic convex quadratic programs with linear constraints in the next section. We note that our techniques which are based on 
counting solutions to polynomial systems were employed for a similar purpose in \cite{DMS05}. 

First we consider the polynomial system obtained by clearing 
denominators and dropping the $x \geq 0$ condition in \eqref{eq:KKT_conditions_LP}:
\begin{equation}\label{eq:KKT_LP_cleared}
    \begin{split}
        &c_ix_i-\lambda - (y^ta_i)x_i =0, \quad i=1,\ldots, m\\
        & Ax=b
    \end{split}
\end{equation}
where $a_i$ is the $i$th column of the matrix $A$. For generic data,
the central curve is obtained as the Zariski closure in $\C^m$
of the projection of the solution set in $(\C^*)^{m+d+1}$ to 
the equations \eqref{eq:KKT_LP_cleared}. Further, the degree of this
central curve would be equal to the number of points in $(\C^*)^m$
obtained as the intersection of the central curve with a generic
hyperplane defined by $ex = f$. 
\begin{lemm} \label{lemm:LP}
The degree of $\mathcal{C}_{LP}(c,A,b)$ for generic $c$, $A$, and $b$ is equal to the number of solutions in $(\C^*)^{m+d+1}$ to the system \eqref{eq:KKT_LP_cleared}
together with an extra equation of the form $ex = f$ where the coefficients of this equation are generic.
\end{lemm}
\begin{proof} Clearly, every solution 
to \eqref{eq:KKT_LP_cleared} plus $ex = f$ in $(\C^*)^{m+d+1}$ projects to a point in $\mathcal{C}_{LP}(c, A,b) \cap \{x \, : ex = f\}$. Conversely, the genericity
of $ex=f$ implies that the points in  $\mathcal{C}_{LP}(c,A,b) \cap \{x \, : ex = f\}$ come from points in $(\C^*)^{m+d+1}$ that satisfy 
\eqref{eq:KKT_LP_cleared} and 
$ex = f$. We show that for each point $x^*$ "downstairs" there is a unique point "upstairs". Suppose there are at least two points $(x^*, y^*, \lambda^*)$ and $(x^*, z^*, \mu^*)$ with these 
properties. Then it is easy to check that $(x^*, ty^* + (1-t)z^*, t\lambda^* + (1-t)\mu^*)$ is also a solution with the same properties 
for any $t$. But this is a contradiction since we have only finitely many preimages by the genericity of the data.
\end{proof}
This lemma implies that in order to compute the 
degree of $\mathcal{C}_{LP}(c,A,b)$ for generic $A,c$, and $b$
we need to count the solutions in $(\C^*)^{m+d+1}$ to
\begin{equation}\label{eq:KKT_LP_torus}
    \begin{split}
        &c_ix_i-\lambda - (y^ta_i)x_i =0, \quad i=1,\ldots, m\\
        & Ax=b \\
        & ex = f
    \end{split}
\end{equation}
where $ex = f$ is also generic. Note that the rank of the matrix 
$\left( \begin{array}{c} A \\ e \end{array} \right)$ is $d+1$
and the solutions to the last $d+1$ equations in \eqref{eq:KKT_LP_torus} can be parametrized by
$$ x = v_0 + t_1v_1 + \cdots + t_{m-d-1}v_{m-d-1} $$
where $v_0, v_1, \ldots, v_{m-d-1}$ are generic vectors. 
Substituting this into the first $m$ equations in \eqref{eq:KKT_LP_torus} we obtain $m$ equations in 
$m$ variables $\lambda, y_1, \ldots, y_d, t_1, \ldots, t_{m-d-1}$.
Furthermore, the genericity assumptions guarantee that 
each equation will have support equal to 
$$\lambda, 1, t_1,  \ldots, t_{m-d-1}, y_1, y_1t_1, \ldots, y_1 t_{m-d-1}, \ldots, y_d, y_dt_1, \ldots, y_dt_{m-d-1}.$$
The Newton polytope of a polynomial with this support is 
a pyramid of height one with base equal to the product of 
simplices $\Delta_{m-d-1} \times \Delta_d$. 
\begin{thm} \label{LP-volume}
$\psi_{LP}(m,d)$ is equal to the volume of 
$\Delta_{m-d-1} \times \Delta_d$:
$$ {m-1 \choose d} \, = \, \sum_{k=0}^{m-d-1}  {m-k-2 \choose d-1}$$
\end{thm}
\begin{proof}
The above lemma and the previous discussion imply that $\psi_{LP}(m,d)$
is equal to the number solutions in $(\C^*)^m$ to $m$ equations in 
$m$ variables, where each equation has support equal 
to the set of monomials listed above. Bernstein's Theorem implies
that this number is bounded above by the normalized volume of the Newton
polytope of these monomials. Since this polytope is a pyramid of height
one over $\Delta_{m-d-1} \times \Delta_d$, we just need to compute
the normalized volume of the product of simplices. Further, because every triangulation of $\Delta_{m-d-1} \times \Delta_d$ is unimodular
we just need to count the number of simplices in any triangulation.
One such triangulation is the staircase triangulation. The
maximal simplices in this triangulation are described as follows.
Consider a $(m-d) \times (d+1)$ rectangular grid. The simplices
in the staircase triangulation of $\Delta_{m-d-1} \times \Delta_d$
are in bijection with paths from the northwest corner of this grid
to the southeast corner where a path consists of steps in the east
or south direction. The total number of steps in each path is $m-1$, and out of these steps $d$ have to be south steps. Therefore there
are a total of $\binom{m-1}{d}$ such paths. These paths 
can be partitioned into those which reach the south edge 
of the grid $k$ steps before the southeast corner where $k=0, \ldots, m-d-1$. The number of these kinds of paths for each $k$ is $\binom{m-k-2}{d-1}$. Finally, the proof of Lemma 11 in \cite{DSV12} implies that $\psi_{LP}(m,d) \geq \binom{m-1}{d}$, and this concludes the proof. 
\end{proof}

\section{Quadratic Programs} \label{sec:QP}
To complete our study of central curves in optimization problems
we will now consider convex quadratic programs with linear constraints.
\begin{equation}\label{eq:QP_primal}
    \begin{split}
        \text{minimize } & \frac{1}{2} x^tQx + cx\\
        \text{subject to } & Ax=b\\
        & x \geq 0,
    \end{split}
\end{equation}
where $Q$ is an $m \times m$ positive definite matrix, $c \in \R^m$, $A$ is $d \times m$ matrix of rank $d$, and $b \in \R^d$. The KKT
conditions that lead to the definition of the central curve are
\begin{equation}\label{eq:KKT_conditions_QP}
    \begin{split}
        &x^t Q + c -\lambda \left(\frac{1}{x_1},\ldots, \frac{1}{x_m}\right)- y^tA =0\\
        & Ax=b, \\
        & x \geq 0.
    \end{split}
\end{equation}
When $Q$, $c$, $A$, and $b$ are generic, we denote by $\psi_{QP}(m,d)$
the degree of the central curve for generic quadratic programs. One can
show by a homotopy continuation argument that it is sufficient
to assume $Q$ to be a generic diagonal matrix. For precise details of this result, we refer the reader to \cite[Section 3.2]{Schlief}. 
With $Q = \mathrm{diag}(q_1,\ldots, q_m)$, after clearing denominators 
and ignoring the nonnegativity constraints $x \geq 0$ in \eqref{eq:KKT_conditions_QP}, we arrive to the following system of 
polynomial equations:
\begin{equation}\label{eq:KKT_QP_cleared}
    \begin{split}
        &q_i x_i^2 + c_ix_i - \lambda - (y^ta_i)x_i = 0 \quad i=1, \ldots, m \\
        & Ax=b,
    \end{split}
\end{equation}
where $a_i$ is the $i$th column of the matrix $A$. As in the linear programming case we have the following lemma. 
\begin{lemm} \label{lemm:QP}
$\psi_{QP}(m,d)$, the degree of the central curve of a generic quadratic program is equal to the number of solutions in $(\C^*)^{m+d+1}$ to the system \eqref{eq:KKT_QP_cleared}
together with an extra equation of the form $ex = f$ where the coefficients of this equation are also generic.
\end{lemm}
\begin{proof}
The proof of this lemma is identical to the proof of Lemma \ref{lemm:LP}.
\end{proof}
\begin{thm} \label{QP-volume}
$$  \psi_{QP}(m,d) \leq \sum_{k=0}^{m-d-1}  {m-k-2 \choose d-1} 2^k .$$
This is the volume of the Newton polytope of a polynomial with support in monomials
$$ \lambda, 1, t_1,  \ldots, t_{m-d-1}, t_1^2, t_1t_2, \ldots,  t_{m-d-1}^2$$
$$ y_1, y_1t_1, \ldots, y_1 t_{m-d-1}, \ldots, y_d, y_dt_1, \ldots, y_dt_{m-d-1}$$
\end{thm}
\begin{proof} 
By Lemma \ref{lemm:QP} and as in the proof of Theorem \ref{LP-volume} we need to count solutions to 
\eqref{eq:KKT_QP_cleared} plus a generic linear equation $ex = f$
in the torus $(\C^*)^{m+d+1}$. The solutions to the equations $Ax = b$
and $ex = f$ can again be parametrized as
$$ x = v_0 + t_1v_1 + \cdots + t_{m-d-1}v_{m-d-1} $$
where $v_0, \ldots, v_{m-d-1}$ are generic vectors. 
Substituting this into the first $m$ equations in \eqref{eq:KKT_QP_cleared} we obtain $m$ equations in 
$m$ variables $\lambda, y_1, \ldots, y_d, t_1, \ldots, t_{m-d-1}$.
Furthermore, the genericity assumptions guarantee that 
each equation will have support equal to 
$$ \lambda, 1, t_1,  \ldots, t_{m-d-1}, t_1^2, t_1t_2, \ldots,  t_{m-d-1}^2$$
$$ y_1, y_1t_1, \ldots, y_1 t_{m-d-1}, \ldots, y_d, y_dt_1, \ldots, y_dt_{m-d-1}$$
The number of solutions to these $m$ equations in $(\C^*)^m$ is
bounded by the normalized volume of the Newton polytope of the above monomials.
Since this is a pyramid of height one, we just need to compute 
the volume of the Newton polytope of the monomials except $\lambda$. 
This polytope has a staircase triangulation as for $\Delta_{m-d-1} \times \Delta_d$ where each simplex corresponds to a path
as we described in the proof of Theorem \ref{LP-volume}, except
that the volume of a simplex corresponding to a path which reaches 
the south edge of the grid $k$ steps before the southeast corner is 
$2^k$. Therefore $\psi_{QP}(m,d)$ is at most 
$$\sum_{k=0}^{m-d-1}  {m-k-2 \choose d-1} 2^k .$$
\end{proof}

\section{Acknowledgements}
We are grateful to Bernd Sturmfels for his help in Lemma \ref{lem:codim-2},
and to Frank Sottile for pointing out an error in the original version of Theorem \ref{QP-volume}. 

\vspace{-2cm}
\bibliography{references}
\end{document}